\documentclass[a4paper,12pt]{article}
\usepackage{amsmath,amsthm,amssymb,amsfonts}
\usepackage{bbm,bm}
\usepackage{latexsym}
\usepackage{mathrsfs}
\usepackage{threeparttable}
\usepackage{tabularx}
\usepackage{booktabs}
\usepackage{graphicx,subfigure}
\usepackage{color}
\usepackage{indentfirst}
\usepackage{geometry}
\usepackage{caption}
\usepackage{lineno}
\usepackage{abstract}
\usepackage{enumerate,mdwlist}
\usepackage[numbers,sort&compress]{natbib}
\usepackage[colorlinks,linkcolor=blue,citecolor=red]{hyperref}
\usepackage{epstopdf}
\usepackage{pifont}
\usepackage{chngcntr}

\def \[{\begin{equation}}
\def \]{\end{equation}}

\newtheorem{thm}{Theorem}[section]
\newtheorem{prop}[thm]{Proposition}
\newtheorem{defi}[thm]{Definition}
\newtheorem{claim}{Claim}[thm]\counterwithout{claim}{section}
\newtheorem{case}{Case}[thm]\counterwithout{case}{section}
\newtheorem{subcase}{Subcase}[case]
\newtheorem{lem}[thm]{Lemma}

\newtheorem{conj}[thm]{Conjecture}

\begin{document}

\setlength{\baselineskip}{20pt}
\begin{center}

{\Large \bf The minimum degree of minimal 2-extendable
claw-free graphs$^{\text{\ding{73}}}$}

\vspace{4mm}

{Jing Guo$^{\rm 1, 2}$, Fuliang Lu$^{\rm 3}$, Heping Zhang$^{\rm 1 \ast}$}

\vspace{4mm}

\footnotesize{$^{\rm 1}$School of Mathematics and Statistics, Lanzhou University,
Lanzhou, 730000, PR China}

\footnotesize{$^{\rm 2}$School of Science, Lanzhou University of
Technology, Lanzhou, 730050, PR China}

\footnotesize{$^{\rm 3}$School of Mathematics and Statistics,
Minnan Normal University, Zhangzhou, 363000, PR China}
\renewcommand\thefootnote{}
\footnote{$^{\text{\ding{73}}}$ This work is supported by NSFC\,
(Grant No. 12271229 and 12271235)
and NSF of Fujian Province (Grant No. 2021J06029).}

\footnote{$^{\ast}$ The corresponding author. \\E-mail addresses:
guoj@lut.edu.cn (J. Guo), flianglu@163.com (F. Lu)
and zhanghp@lzu.edu.cn (H. Zhang).}

\end{center}

\noindent {\bf Abstract}:
A connected graph $G$ with a perfect matching is said to be $k$-extendable for integers $k$,
$1 \leq k\leq \frac{|V(G)|}{2}-1$, if any matching in $G$ of size $k$ is contained
in a perfect matching of $G$.
A $k$-extendable graph is minimal if the deletion of any edge results in a graph
that is not $k$-extendable. In 1994,
Plummer proved that every $k$-extendable claw-free graph
has minimum degree at least $2k$.
Recently, He et al. showed that every minimal 1-extendable graph
has minimum degree 2 or 3.
In this paper, we prove that the minimum degree of a minimal
2-extendable claw-free graph is either $4$ or $5$.

\vspace{2mm}
\noindent{\bf Keywords}: Perfect matching;
Minimal 2-extendable graph; Claw-free graph; Minimum degree
\vspace{2mm}

\noindent{AMS subject classification:} 05C70,\ 05C07

{\setcounter{section}{0}
\section{Introduction}\setcounter{equation}{0}

Graphs considered in this paper are finite, undirected and simple.
We follow \cite{BM} for undefined notation and terminology.
Let $G$ be a graph with vertex set $V(G)$ and edge set $E(G)$.
The {\em order} of $G$ is the cardinality of $V(G)$
and its {\em size} is the number of its edges.
A {\em matching} $M$ of $G$ is a subset of $E(G)$ such that any two edges of $M$ have
no end-vertex in common. A matching of size $k$ is called a {\em $k$-matching}.
If every vertex of $G$ is incident with exactly one edge of $M$,
then $M$ is called a {\em perfect matching} of $G$.
A connected graph of order at least two is {\em matching covered}
if each of its edges lies in a perfect matching.

A connected graph $G$ with a perfect matching is said to be {\em $k$-extendable}
for integers $k$, $1 \leq k\leq \frac{|V(G)|}{2}-1$, if each $k$-matching of $G$ is contained
in a perfect matching of $G$.
Every matching covered graph is 1-extendable except $K_2$.
Plummer \cite{P} introduced $k$-extendable graphs
which play an elementary role in matching theory
and obtained the following two basic properties of $k$-extendable graphs.

\begin{thm}[\cite{P}]\label{P}
Let $G$ be a graph of order $2n$ and $1\leq k\leq n-1$.
If $G$ is $k$-extendable, then $G$ is $(k-1)$-extendable and $(k+1)$-connected.
\end{thm}

Much attention to the theory of matching extension has been paid.
We refer the reader to Lov\'{a}sz and Plummer's book \cite{LP} for
an excellent survey of the matching theory, and \cite{PLM,YL} for recent progress.

A matching covered graph which is free of nontrivial tight cuts is
a {\em brick} if it is nonbipartite and a {\em brace} if it is bipartite.
Edmonds et al. \cite{ELW} and Lov\'{a}sz \cite{LO} proposed and developed
the ``tight cut decomposition" of matching covered graphs into
list of bricks and braces in an essentially unique manner.
This decomposition reduces several problems in matching covered graphs to
bricks and braces (for example, a matching covered graph is Pfaffian if and
only if its bricks and braces are Pfaffian \cite{VVM}).
Plummer \cite{P} provided a connection between the 2-extendability
and bricks and braces.

\begin{thm}[\cite{P}]\label{PP}
Let $G$ be a $2$-extendable graph. Then $G$ is either a brick or a brace.
\end{thm}

We refer the reader to \cite{MHCH} for more about bricks and braces.

A $k$-extendable graph is {\em minimal} if the deletion of any edge
results in a graph that is not $k$-extendable.
The {\em degree} of a vertex $u$ in $G$, denoted by $d_{G}(u)$,
is the number of edges incident with $u$.
Denote by $\delta(G)$ the {\em minimum degree} of $G$,
which is the minimum value in degrees of all vertices of $G$.
Anunchuen and Caccetta \cite{AC1992} proved that, apart from the complete
graph $K_{2n}$, every minimal $k$-extendable graph of order $2n$
has minimum degree at most $n+k-1$ and \cite{AC1997} obtained the following result.

\begin{thm}[\cite{AC1997}]\label{AC1997}
Let $G$ be a minimal $k$-extendable graph of order $2n$ and $1\leq k\leq n-1$.
Then either $k+1\leq \delta(G)\leq n$ or $\delta(G)\geq 2k+1$.
\end{thm}

Moreover, Lou \cite{L} showed that the minimum degree of
a minimal $k$-extendable bipartite graph is $k+1$.
Lou and Yu \cite{LY} conjectured that every minimal $k$-extendable graph
of order $2n$ with $n\leq 2k$ has minimum degree $k+1, 2k$ or $2k+1$.

A graph containing no induced subgraph isomorphic to the complete
bipartite graph $K_{1,3}$ is said to be {\em claw-free}.
Claw-free graphs have been a popular type of graphs in various graph structures
(for example, 1-factor and Hamilton cycle; see \cite{CS}),
which may involve all possible values of $k$ from 1 to $n-2$.
Plummer \cite{PL} obtained a lower bound of the minimum degree of $k$-extendable claw-free graphs.

\begin{thm}[\cite{PL}]\label{PL}
Let $G$ be a $k$-extendable claw-free graph of order $2n$ and $1\leq k\leq n-1$.
Then $\delta(G)\geq 2k$.
\end{thm}



Recently, Zhang et al. \cite{ZWY} characterized minimal 1-extendable
claw-free graphs.
He et al. \cite{HLX} showed that every minimal 1-extendable graph has
minimum degree 2 or 3.
In this paper, we determine the minimum degree of minimal
2-extendable claw-free graphs.
The following theorem is our main result.

\begin{thm}\label{main}
Let $G$ be a minimal $2$-extendable claw-free graph.
Then $\delta(G)=4$ or $5$.
\end{thm}

Some preliminaries are presented in Section 2.
The proof of Theorem \ref{main} will be given in Section 3.

\section{Preliminaries}

We begin with some notations.
For a vertex $u$ of a graph $G$,
let $N_{G}(u)$ (or simply $N(u)$) denote the set of neighbors of vertex $u$ in $G$.
Denote by $G-uv$ the graph obtained from $G$ by deleting the edge $uv$
if $uv \in E(G)$.
Similarly, $G+uv$ stands for the graph obtained from $G$ by adding an
edge $uv$ if $uv \notin E(G)$.
For $X, Y \subseteq V(G)$, by $E_{G}(X, Y)$ (or simply $E(X, Y)$) we mean the set of edges
of $G$ with one end-vertex in $X$ and the other end-vertex in $Y$;
by $G[X]$ we mean the subgraph of $G$ induced by $X$;
by $M_G(X)$ (or simply $M(X)$) we mean a maximum matching in $G[X]$.
An {\em independent set} of $G$ is a set of pairwise nonadjacent vertices in $G$.
A {\em complete graph} is a graph in which any two vertices are adjacent.

Let $G$ be a connected graph and $k$ be a positive integer
with $k<|V(G)|$. A {\em $k$-vertex cut} of $G$ is
a subset of $V(G)$ with cardinality $k$ whose removal disconnects $G$,
and $G$ is {\em $k$-connected}
if $G$ contains no vertex cut with cardinality less than $k$.
Similarly, $G$ is said to be {\em $k$-edge-connected}
if the deletion of less than $k$ edges from $G$ does not disconnect it.
We call a graph $G$ {\em trivial} if $|V(G)|=1$,
and {\em nontrivial} otherwise.



The following characterization of Tutte's type of minimal $k$-extendable
graphs due to Anunchuen and Caccetta \cite{AC1994} will be useful
in the proof of Theorem \ref{main}.
As usual, let $c_{o}(G)$ denote the number of odd components of a graph $G$.

\begin{thm}[\cite{AC1994}]\label{AC1994}
Let $G$ be a $k$-extendable graph of order $2n$
and $1 \leq k \leq n-1$.
Then $G$ is minimal if and only if for every edge $e=uv \in E(G)$,
there exists $S_{e} \subseteq V(G)\backslash \{u, v\}$
such that the following statements hold$:$

{\rm (\romannumeral1)} $|M(S_{e})| \geq k;$

{\rm (\romannumeral2)} $c_{o}(G_{e})=|S_{e}|-2k+2$, where $G_e=G-e-S_e;$

{\rm (\romannumeral3)} $u$ and $v$ belong to two odd components of $G_{e}$.
\end{thm}

Next we present some useful lemmas. Let $G$ be a graph with a perfect matching.
If $G$ has a matching $M$ with size at most $k$ and $G-V(M)$ has an odd component,
then there exists a $k$-matching containing $M$ and the $k$-matching
is not contained in any perfect matching of $G$.
For convenience, we state it as a lemma.

\begin{lem}\label{next1}
Let $G$ be a graph with a perfect matching and
$M$ be a matching of $G$ with size at most $k$.
If $G-V(M)$ has an odd component, then $G$ is not $k$-extendable.
\end{lem}

\begin{lem}\label{next2}
Let $G$ be a graph with a perfect matching and
$X$ be a $(2k-1)$-vertex cut of $G$ such that $|M(X)|=k-1$, $k\geq 2$.
Let $H_1$ and $H_2$ be two components of $G-X$ and
$X\backslash V(M(X))=\{x\}$.
If $N(x)\cap V(H_i)\neq \emptyset$ for $i=1, 2$,
then $G$ is not $k$-extendable.
\end{lem}

\begin{proof}
Let $x_i\in N(x)\cap V(H_i)$, $i=1, 2$.
If $|V(H_1)|$ is odd, then $M(X)\cup \{xx_2\}$ is a $k$-matching of $G$
and the removal of it results in an odd component $H_1$.
If $|V(H_1)|$ is even, then $M(X)\cup \{xx_1\}$ is a $k$-matching
and the removal of it results in an odd component $H_1-x_1$.
In both alternatives, $G$ is not $k$-extendable by Lemma \ref{next1}.
\end{proof}

\begin{lem}\label{next3}
Let $G$ be a $k$-extendable graph and $X$ be a vertex cut of $G$.
If $G[X]$ has a perfect matching $M$, then $|M|\geq k$.
Moreover, if $|M|=k$, then every component of $G-X$ is even and every vertex in $X$
has a neighbor in each even component of $G-X$.
\end{lem}

\begin{proof}
Suppose to the contrary that $|M|<k$.
If $|M|=1$, then $k\geq 2$ and $|X|=2$. Thus $X$ is a 2-vertex cut of $G$,
contradicting that $G$ is $(k+1)$-connected. So $|M|\geq 2$.
Let $H_1$ be a component of $G-X$.
If $|V(H_1)|$ is odd, then $G-V(M)$ has an odd component $H_1$.
By Lemma \ref{next1},
$G$ is not $k$-extendable, a contradiction. So $|V(H_1)|$ is even.

Take any edge $x_1y_1$ in $M$.
Suppose that one of $x_1$ and $y_1$ has a neighbor in $H_1$.
Without loss of generality,
let $x_{1}'\in N(x_1)\cap V(H_1)$ and $N(y_1)\cap V(H_1)=\emptyset$.
Then $(M\backslash \{x_1y_1\})\cup \{x_1x_{1}'\}$
is a matching with size at most $k-1$ and the removal of it results in an odd component
$G[V(H_1)\backslash \{x_{1}'\}]$. By Lemma \ref{next1}, $G$ is not $k$-extendable, a contradiction.
Thus both $x_1$ and $y_1$ have neighbors in $H_1$
or none of them have neighbors in $H_1$.

Since $G$ is $(k+1)$-connected and $|X|<2k$,
there exists $x_2y_2\in M\backslash \{x_1y_1\}$ such that $x_2$
(resp. $y_2$) has neighbors in both $H_1$ and $G-X-V(H_1)$.
Let $x_{2}'\in N(x_2)\cap V(H_1)$ and $y_{2}'\in N(y_2)\cap V(G-X-V(H_1))$.
Then $(M\backslash \{x_2y_2\})\cup \{x_2x_{2}', y_2y_{2}'\}$
is a matching with size at most $k$ and the removal of it results in
an odd component $G[V(H_1)\backslash \{x_{2}'\}]$. By Lemma \ref{next1}, $G$ is not $k$-extendable,
a contradiction. So $|M|\geq k$.

If $|M|=k$, then $G-V(M)$ has a perfect matching as $G$ is $k$-extendable.
Thus every component of $G-V(M)$ is even. By similar discussions above,
the two end-vertices of every edge in $M$ have neighbors in
each even component of $G-V(M)$.
\end{proof}

\begin{lem}\label{next4}
Let $G$ be a minimal $k$-extendable graph.
For any edge $e=uv\in E(G)$, let $S_e$, $M(S_e)$ and $G_e$
be defined in Theorem $\ref{AC1994}$.
Let $M\subseteq M(S_e)$ and $|M|=k$. Then every perfect matching
of $G-V(M)$ contains $e$. Moreover, if $|S_e|\geq 2k+1$,
then every vertex in $S_e\backslash V(M)$ has a neighbor in some
odd component of $G_e$ other than the odd components containing $u$ and $v$.
\end{lem}

\begin{proof}
Let $M'$ be a perfect matching of $G-V(M)$.
By Theorem \ref{AC1994},
$|S_e\backslash V(M)|=|S_e|-2k=c_{o}(G_{e})-2$
and $e$ joins two odd components of $G_e$.
Then $e \in M'$ and every vertex in $S_e\backslash V(M)$
(if $|S_e|\geq 2k+1$) is matched to a vertex
in some odd component of $G_e$ other than the odd components
containing $u$ and $v$. The result follows.
\end{proof}

\section{Proof of Theorem \ref{main}}

In this section, we prove our main theorem.
We first present an useful lemma.

\begin{lem}\label{lem1}
Let $G$ be a minimal $2$-extendable claw-free graph with $\delta(G)\geq 6$.
For any edge $e=uv \in E(G)$, let $S_e$, $M(S_e)$ and $G_e$ be defined in Theorem $\ref{AC1994}$
and $S_e$ be a smallest set satisfying the conclusions of Theorem $\ref{AC1994}$
when Theorem $\ref{AC1994}$ is applied to the edge $e$. Let $O_1, O_2, \ldots, O_t$ be
the odd components of $G_e$ such that $u \in V(O_1)$ and $v \in V(O_2)$.
Then $|M(S_e)|=2$ and one of the following statements holds$:$

{\rm (\romannumeral1)} $|S_e|=4$, $t=2$ and $|V(O_i)|>1$ for $i=1,2;$

{\rm (\romannumeral2)} $|S_e|=5$, $t=3$, $|V(O_i)|=1$ for $i=1,2$
and $G_e$ has no even component$;$
moreover, for any vertex $w\in S_e\backslash V(M(S_e))$,
$\{uw, vw\}\subseteq E(G)$ and $N(w)\cap V(O_3)\neq \emptyset$.
\end{lem}

\begin{proof}
By Theorem \ref{AC1994}, $|M(S_e)|\geq 2$ and $t=|S_e|-2$.
We begin the proof by making the following claim.

\begin{claim}\label{m1}
Let $M\subseteq M(S_e)$ and $|M|=2$. If $|S_e|\geq 5$,
then, for any vertex $w\in S_e\backslash V(M)$,
$w$ satisfies the property that
exactly three odd components of $G_e$ contain vertices in $N(w)$,
$\{uw, vw\}\subseteq E(G)$ and $N(w)\cap V(O_i)\neq \emptyset$ for some $i\geq 3$.
In particular, if $|M(S_e)|\geq 3$, then every vertex in $S_e$
satisfies the property.
\end{claim}

\noindent Proof of Claim \ref{m1}.
By Lemma \ref{next4}, $N(w)\cap V(O_i)\neq \emptyset$ for some $i\geq 3$.
We next show that exactly three odd components of $G_e$ contain vertices
in $N(w)$ and $\{uw, vw\}\subseteq E(G)$.

Suppose that $w$ has neighbors in at most two odd components of $G_e$.
Since $N(w)\cap V(O_i)\neq \emptyset$ for some $i\geq 3$,
at most one of $O_1$ and $O_2$ contains vertices in $N(w)$.
Let $S_e'=S_e \backslash \{w\}$. Adding $w$ back to the graph $G_e$ either
turns an odd component into an even component (if $w$ has neighbors in
one odd component) or merges two odd components into an
odd component (if $w$ has neighbors in two odd components).
Thus $c_o(G-e-S_e')=c_o(G-e-S_e)-1=|S_e|-2-1=|S_e'|-2$
and $u$ and $v$ belong to two odd components of $G-e-S_e'$.
Then, for the edge $e$, $S_e'$ and $G-e-S_e'$
satisfy the conclusions of Theorem \ref{AC1994}.
But $|S_e'|<|S_e|$, a contradiction to the minimality of $S_e$.

Suppose that $w$ has neighbors in at least four odd components of $G_e$.
Since the only edge between different odd components of $G_e$ is $e$,
at least three components of $G_e+e$ contains vertices in $N(w)$.
Let three of them be $w_1, w_2$ and $w_3$.
Then $G[\{w, w_1, w_2, w_3\}]$ is a claw, a contradiction.
Finally, if $w$ has neighbors in exactly three
odd components but $uw\notin E(G)$ or $vw\notin E(G)$,
then $G$ also contains a claw, a contradiction.

Noticing that $M$ is an arbitrary $2$-matching in $G[S_e]$,
every vertex in $S_e$ satisfies the property whenever $|M(S_e)|\geq 3$.
\hfill $\square$

\begin{claim}\label{m3}
$|M(S_e)|=2$.
\end{claim}

\noindent Proof of Claim \ref{m3}.
Suppose to the contrary that $|M(S_e)|\geq 3$.
We construct a graph $G'$ by contracting every $O_i$ into a
single vertex and deleting all the multiple edges appeared in this process.
Let $U_{e}$ denote the set of the vertices resulting from the contraction of all $O_i's$.
Then $|U_{e}|=|S_{e}|-2$ and $V(G')=S_{e}\cup U_{e}$.
For convenience, we use $u'$ and $v'$ to represent
the vertices resulting from the contractions of $O_1$ and $O_2$, respectively.
Since $G$ is $3$-connected, each $O_i$ is connected to at least $3$ vertices in
$V(G) \backslash V(O_i)$, $i=1, 2, \ldots, t$.
It follows that every vertex in $U_{e} \backslash \{u', v'\}$ has
at least $3$ neighbors in $S_{e}$.
Then there are at least $3 (|U_{e}|-2)$
edges from $U_{e} \backslash \{u', v'\}$ to $S_{e}$ in $G'$.

Conversely, by Claim \ref{m1}, every vertex in $S_{e}$ has a neighbor in
exactly one $O_{i}$ for some $i \geq 3$. Then there are exactly $|S_{e}|$ edges
from $S_{e}$ to $U_{e} \backslash \{u', v'\}$ in $G'$.
So we estimate the number of the edges between $S_{e}$
and $U_{e} \backslash \{u', v'\}$ as follows:
\begin{center}
$|S_{e}|=|E(S_{e}, U_{e} \backslash \{u', v'\})| \geq 3 (|U_{e}|-2)=3 (|S_{e}|-4)$,
\end{center}
which implies that $|S_{e}| \leq 6$.
By the hypothesis that $|S_{e}| \geq 6$, $|S_{e}|=6$.
So $|M(S_e)|=3$ and $t=c_o(G_e)=4$.

We show that $|V(O_1)|=1$ and $|V(O_2)|=1$.
By symmetry, suppose that $|V(O_1)|>1$.
If there exists a vertex $x$ in $S_e$
such that $N(x) \cap (V(O_{1}) \backslash \{u\})\neq \emptyset$, then
let $x_1\in N(x)\cap (V(O_{1}) \backslash \{u\})$.
By Claim \ref{m1}, $vx\in E(G)$ and $N(x) \cap V(O_i)\neq \emptyset$
for some $i \geq 3$.
Let $x_2\in N(x) \cap V(O_i)$.
Then $G[\{x, x_1, x_2, v\}]$ is a claw, a contradiction.
So none of the vertices in $S_{e}$ have neighbors in $V(O_{1}) \backslash \{u\}$.
It implies that $\{u\}$ is a vertex cut of $G$, contradicting that $G$ is $3$-connected.

If there exists an edge $x_1y_1$ in $M(S_e)$ such that
$N(x_1)\cap V(O_3)\neq \emptyset$ and $N(y_1)\cap V(O_3)\neq \emptyset$,
then $N(x_1)\cap V(O_4)=\emptyset$ and $N(y_1)\cap V(O_4)=\emptyset$
by Claim \ref{m1}.
Thus $M(S_e)\backslash \{x_1y_1\}$ is a $2$-matching and the removal of it
results in an odd component $O_4$. By Lemma \ref{next1},
$G$ is not $2$-extendable, a contradiction.
Then it is impossible that the two end-vertices of any edge in $M(S_e)$
have neighbors in $O_3$ (resp. $O_4$).

It follows that
if there exist edges $\{x_1y_1, x_2y_2\}\subseteq M(S_e)$ such that
$N(x_1)\cap V(O_3)\neq \emptyset$ and
$N(x_2)\cap V(O_3)\neq \emptyset$,
then $N(y_1)\cap V(O_3)=\emptyset$ and $N(y_2)\cap V(O_3)=\emptyset$.
Suppose that $x_1x_2\in E(G)$.
Then $(M(S_e)\backslash \{x_1y_1, x_2y_2\})\cup \{x_1x_2\}$ is a $2$-matching and
the removal of it results in an odd component $O_3$.
By Lemma \ref{next1}, $G$ is not $2$-extendable, a contradiction.

\begin{figure}[h]
\centering
\includegraphics[height=5cm,width=5cm]{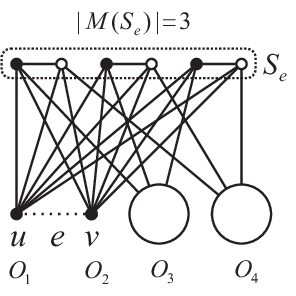}
\caption{\label{tu-m3} All the black (resp. white) vertices in $S_e$ form the
independent set $A$ (resp. $B$).}
\end{figure}

From the discussions in the previous two paragraphs, we deduce that all the vertices in $S_e$
that have neighbors in $O_3$ (resp. $O_4$) form an independent set of $G$, say $A$ (resp. $B$)
(as shown in Fig. \ref{tu-m3}).
By Claim \ref{m1}, every vertex in $A$ is adjacent to $u$.
Since $|A|=|M(S_e)|=3$, $G[A\cup \{u\}]$ is a claw, a contradiction.
\hfill $\square$

\vspace{2mm}
By Claim \ref{m3}, we assume that $M(S_e)=\{x_1y_1, x_2y_2\}$ in the following proof.

\begin{claim}\label{m4} If $|S_e|=5$ or $6$, then, for every edge in $M(S_e)$,
at least one end-vertex of it has a neighbor in $O_i, i=3, 4, \ldots, t$.
\end{claim}

\noindent Proof of Claim \ref{m4}.
If $|S_e|=6$, then $t=4$.
Let $\{a_1, a_2\}=S_e \backslash V(M(S_e))$. By Claim \ref{m1},
we may assume that $N(a_1)\cap V(O_3)\neq \emptyset$ and
$N(a_2)\cap V(O_4)\neq \emptyset$.
Then $N(a_1)\cap V(O_4)=\emptyset$ and $N(a_2)\cap V(O_3)=\emptyset$
as $G$ is claw-free. Without loss of generality,
suppose that $N(x_1)\cap V(O_3)=\emptyset$ and $N(y_1)\cap V(O_3)=\emptyset$.
Then $\{x_2y_2, ua_1\}$ is a $2$-matching and
the removal of it results in an odd component $O_3$.
By Lemma \ref{next1}, $G$ is not $2$-extendable, a contradiction.
If $|S_e|=5$, then the proof is similar to the case of $|S_e|=6$.
\hfill $\square$

\begin{claim}\label{m6}
$|S_e|=4$ or $5$.
\end{claim}

\noindent Proof of Claim \ref{m6}.
Since $|M(S_e)|\geq 2$, $|S_e|\geq 4$. Suppose to the contrary that $|S_e|\geq 6$.
We divide the proof into the two cases $|S_e|\geq 7$ and $|S_e|=6$.

If $|S_e|\geq 7$, then, by Claim \ref{m3}, the set of the vertices
in $S_e\backslash V(M(S_e))$ is an independent set of $G$.
By Claim \ref{m1}, every vertex in $S_e\backslash V(M(S_e))$
is adjacent to $u$. Then $G[\{u\}\cup (S_e\backslash V(M(S_e)))]$
contains a claw, a contradiction.

If $|S_e|=6$, then $t=4$.
Let $\{a_1, a_2\}=S_e\backslash V(M(S_e))$.
Then $\{a_1, a_2\}$ is an independent set of $G$.
By Claim \ref{m1}, exactly three odd components of $G_e$ contain vertices in $N(a_1)$
(resp. $N(a_2)$) and $\{ua_1, va_1\}\subseteq E(G)$
(resp. $\{ua_2, va_2\}\subseteq E(G)$).
We may assume that
$N(a_1)\cap V(O_3)\neq \emptyset$, $N(a_1)\cap V(O_4)=\emptyset$,
$N(a_2)\cap V(O_3)=\emptyset$ and $N(a_2)\cap V(O_4)\neq \emptyset$.
Since $G$ is claw-free, $|V(O_1)|=1$ and $|V(O_2)|=1$.

We next show that every vertex in $V(M(S_e))$ has a neighbor in one of $O_3$
and $O_4$.

Without loss of generality, suppose
that $N(x_1)\cap V(O_3)\neq \emptyset$ and $N(x_1)\cap V(O_4)\neq \emptyset$.
Then $ux_1\not\in E(G)$ and $vx_1\not\in E(G)$ as $G$ is claw-free.
Since $\delta(G)\geq 6$, $u$ and $v$ are adjacent to every vertex
in $S_e\backslash \{x_1\}$. So $a_1y_1\in E(G)$ or $a_2y_1\in E(G)$.
Otherwise, $G[\{u, a_1, a_2, y_1\}]$ is a claw,
a contradiction. Let $a_1y_1\in E(G)$ and $S_{e}'=S_e\backslash \{x_1\}$.
Then $M(S_{e}')=\{x_2y_2, a_1y_1\}$,
$c_o(G-e-S_{e}')=$$c_o(G-e-S_{e})-1=$$|S_{e}'|-2$ and
$u$ and $v$ belong to two odd components of $G-e-S_{e}'$. But $|S_{e}'|<|S_e|$,
a contradiction to the minimality of $S_e$.
By Claim \ref{m4}, $O_3$ (resp. $O_4$) is connected to at least one end-vertex of each edge in $M(S_e)$.
Thus, for every edge in $M(S_e)$, one end-vertex has a neighbor in $O_3$
and the other end-vertex has a neighbor in $O_4$.

\begin{figure}[h]
\centering
\includegraphics[height=5.2cm,width=5.4cm]{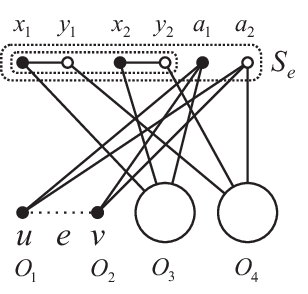}
\caption{\label{tu-c6} All the black (resp. white) vertices in $S_e$
form the independent set $A$ (resp. $B$).}
\end{figure}

Moreover, all the vertices in $S_e$ that have neighbors in $O_3$ (resp. $O_4$)
form an independent set $A$ (resp. $B$) of $G$ (as shown in Fig. \ref{tu-c6}).
Otherwise, let
$N(x_i)\cap V(O_3)\neq \emptyset$, $N(x_i)\cap V(O_4)=\emptyset$,
$N(y_i)\cap V(O_3)=\emptyset$ and $N(y_i)\cap V(O_4)\neq\emptyset$,
$i=1, 2$.
If $x_1x_2\in E(G)$,
then $\{x_1x_2, ua_1\}$ is a $2$-matching and the removal
of it results in an odd component $O_3$.
By Lemma \ref{next1}, $G$ is not $2$-extendable, a contradiction.
Similarly, $a_1x_1\in E(G)$ or $a_1x_2\in E(G)$.
Since $d_G(u)\geq 6$, there is at most one vertex in $S_e$ which is not
adjacent to $u$, say the vertex in $A$.
Then $G[B\cup \{u\}]$ is a claw, a contradiction.
\hfill $\square$

\bigskip
In the sequel, we show that if $|S_e|=4$, then {\rm (\romannumeral1)} holds,
and {\rm (\romannumeral2)} holds otherwise.

\begin{case}\label{case1}
$|S_e|=4$.
\end{case}

In this case, $t=2$. If $|V(O_1)|=1$, then $N(u)\subseteq S_e\cup \{v\}$.
It implies that $d_G(u)\leq 5$, contradicting the hypothesis that $\delta(G)\geq 6$.
Similarly,  $|V(O_2)|>1$. So (\romannumeral1) holds.



\begin{case}\label{case2}
$|S_e|=5$.
\end{case}

Under this condition, $t=3$. Let $S_e\backslash V(M(S_e))=\{a_1\}$.
By Claim \ref{m1}, $\{ua_1, va_1\}\subseteq E(G)$
and $N(a_1)\cap V(O_3)\neq \emptyset$.
We first show that $G_e$ has no even component.

Suppose to the contrary that $G_e$ has an even component $C$.
Since $G$ is claw-free, $N(a_1)\cap V(C)=\emptyset$.
Then $V(M(S_e))$ is a vertex cut of $G$.
By Lemma \ref{next3}, every vertex in $V(M(S_e))$ has a neighbor in $C$.



Let $x_{1}'\in N(x_1)\cap V(C)$ and $y_{1}'\in N(y_1)\cap V(C)$.
By Claim \ref{m4}, $N(x_1)\cap V(O_3)\neq \emptyset$ or
$N(y_1)\cap V(O_3)\neq \emptyset$.
Without loss of generality, let $y_{1}''\in N(y_1)\cap V(O_3)$.
Since $G$ is claw-free, $N(y_1)\cap V(O_1)=\emptyset$ and
$N(y_1)\cap V(O_2)=\emptyset$.
If $|V(O_1)|=1$, then $N(u)\subseteq (S_e\backslash \{y_1\})$ $\cup$ $\{v\}$.
That is, $d_G(v)\leq 5$, a contradiction. Similarly, $|V(O_2)|>1$.

If $x_1$ and $y_1$ have no neighbors in $V(O_1)\backslash \{u\}$,
then $\{x_2y_2, va_1\}$ is a $2$-matching
and the removal of it results in an odd component $O_1$.
By Lemma \ref{next1}, $G$ is not $2$-extendable, a contradiction.
If $x_1$ (resp. $y_1$) has neighbors in both $V(O_1)\backslash \{u\}$
and $V(O_2)\backslash \{v\}$,
then $G$ has a claw, a contradiction.
Thus, for each edge in $M(S_e)$, one end-vertex has a neighbor in $V(O_1)\backslash \{u\}$
and the other end-vertex has a neighbor in $V(O_2)\backslash \{v\}$.
Without loss of generality, let $y_{1}'''\in N(y_1)\cap (V(O_2)\backslash \{v\})$.
Then $G[\{y_1, y_{1}', y_{1}'', y_{1}'''\}]$ is a claw, a contradiction.
Therefore, $G_e$ has no even component.

We next prove that $|V(O_i)|=1$ for $i=1, 2$.
Suppose to the contrary that $|V(O_1)|>1$ or $|V(O_2)|>1$.
By symmetry, we consider the following two subcases.

\begin{subcase}\label{subcase1}
$|V(O_1)|>1$ and $|V(O_2)|=1$ $($as shown in Fig. $\ref{tu-sub1}\ ({\rm a}))$.
\end{subcase}

\begin{figure}[h]
\centering
\includegraphics[height=5.1cm,width=9.9cm]{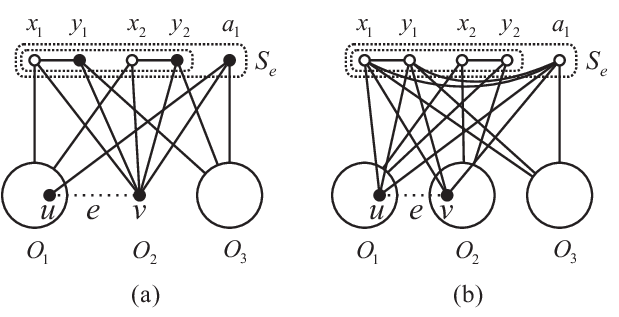}
\caption{\label{tu-sub1} (a) $|V(O_1)|>1$ and $|V(O_2)|=1$; (b) $|V(O_1)|>1$ and $|V(O_2)|>1$.}
\end{figure}

Since every vertex in $S_e$ is adjacent to $v$ and $G$ is claw-free,
every vertex in $V(M(S_e))$ can not have
neighbors in both $V(O_1)\backslash \{u\}$ and $V(O_3)$.
Noticing that $G$ is $3$-connected, $V(O_1)\backslash \{u\}$ is connected to
at least $2$ vertices in $V(M(S_e))$.
By Claim \ref{m4}, at least one end-vertex of each edge in $M(S_e)$
has a neighbor in $V(O_3)$.
Thus, for every edge in $M(S_e)$,
one end-vertex has a neighbor in $V(O_1)\backslash \{u\}$
and the other end-vertex has a neighbor in $V(O_3)$.
Denote by $A$ the set of all vertices in $S_e$ which have neighbors in $V(O_3)$.
Then $A$ is an independent set of $G$.
Otherwise, let $A=\{y_1, y_2, a_1\}$
(as shown in Fig. $\ref{tu-sub1}\ ({\rm a})$).
If $y_1y_2\in E(G)$, then $\{y_1y_2, ua_1\}$ is a $2$-matching and
the deletion of the $2$-matching results in an odd component $O_3$,
which contradicts Lemma \ref{next1}.
But then $G[A\cup \{v\}]$ is a claw, a contradiction.
Similarly, $y_1a_1\in E(G)$ or $y_2a_1\in E(G)$.

\begin{subcase}\label{subcase2}
$|V(O_1)|>1$ and $|V(O_2)|>1$
$($as shown in Fig. $\ref{tu-sub1}\ ({\rm b}))$.
\end{subcase}

Since $G$ is claw-free, $N(a_1)\cap V(O_1)=\{u\}$ and $N(a_1)\cap V(O_2)=\{v\}$.
If $N(a_1)\cap S_e\neq \emptyset$, say $a_1x_1\in E(G)$,
then let $M'(S_{e})=\{a_1x_1, x_2y_2\}$.
By Claim \ref{m1}, $\{uy_1, vy_1\}\subseteq E(G)$
and $N(y_1)\cap V(O_3)\neq \emptyset$.
Since $O_1$ and $O_2$ are nontrivial, $a_1y_1\in E(G)$. Otherwise,
let $u_1\in N(u)\cap V(O_1)$. Then $G[\{u, u_1, a_1, y_1\}]$
is a claw, a contradiction. So $G[\{x_1, y_1, a_1\}]$ is a complete graph
(as shown in Fig. $\ref{tu-sub1}\ ({\rm b})$).
Let $M''(S_{e})=\{a_1y_1, x_2y_2\}$.
By Claim \ref{m1}, $\{ux_1, vx_1\}\subseteq E(G)$
and $N(x_1)\cap V(O_3)\neq \emptyset$.
It follows that $x_1, y_1$ and $a_1$ have no neighbors in $V(O_1)\backslash \{u\}$
(resp. $V(O_2)\backslash \{v\}$) as $G$ is claw-free.
Noticing that $G$ is $3$-connected, $x_2$ and $y_2$
have neighbors in $V(O_1)\backslash \{u\}$ (resp. $V(O_2)\backslash \{v\}$).
By Claim \ref{m4}, $N(x_2)\cap V(O_3)\neq \emptyset$ or $N(y_2)\cap V(O_3)\neq \emptyset$.
Then $G$ has a claw, a contradiction.
So $N(a_1)\cap S_e=\emptyset$.

Finally, let $a_2\in N(a_1)\cap V(O_3)$ and $e_1=a_1a_2$.
Let $S_{e_1}$, $M(S_{e_1})$ and $G_{e_1}$ be defined in Theorem \ref{AC1994}
and $S_{e_1}$ be a smallest vertex set satisfying the
conclusions of Theorem \ref{AC1994}
when Theorem \ref{AC1994} is applied to $e_1$.
By Claims \ref{m3} and \ref{m6}, $|M(S_{e_1})|=2$ and $|S_{e_1}|=4$ or $5$.

\vspace{0.3cm}
\noindent {\bf Subcase 2.2.1. $|S_{e_1}|=4$.}
\vspace{0.3cm}

Let $O_{a_1}$ and $O_{a_2}$ be two odd components of $G_{e_1}$.
Then $|V(O_{a_1})|>1$ and $|V(O_{a_2})|>1$ by (\romannumeral1).
Let $A=N(a_1) \cap N(a_2)$.
Since $d_G(a_1)\geq 6$ and
$N(a_1) \cap (V(G)\backslash V(O_3))=\{u, v\}$,
$|N(a_1) \cap V(O_3)|\geq 4$. So $|N(a_1) \cap N(a_2)|\geq 3$.
Moreover, $A\subseteq S_{e_1}$. Then $|A|=3$ or $4$.
Because $E(\{u, v\}, A)=\emptyset$, $\{u, v\}\cap S_{e_1}=\emptyset$.
Then $N(a_1) \cap V(O_{a_1})=\{u, v\}$. It follows that
if $A$ contains a vertex which has a neighbor in $V(O_{a_1})\backslash \{a_1\}$,
then the vertex has no neighbor in $V(O_{a_2})\backslash \{a_2\}$.
Otherwise, $G$ has a claw, a contradiction.
Since $G$ is $3$-connected, $|E(V(O_{a_1})\backslash \{a_1\}, S_{e_1})|\geq 2$
and $|E(V(O_{a_2})\backslash \{a_2\}, S_{e_1})|\geq 2$.

If $|A|=3$, then let $A=\{b_1, b_2, b_3\}$ and $S_{e_1}=A\cup \{c\}$.
There exist exactly two vertices in $S_{e_1}$ including the vertex $c$
which have neighbors in $V(O_{a_1})\backslash \{a_1\}$
or $V(O_{a_2})\backslash \{a_2\}$.
Without loss of generality, assume that $c$ and $b_1$ have neighbors in
$V(O_{a_1})\backslash \{a_1\}$. Let $c'\in N(c)\cap (V(O_{a_1})\backslash \{a_1\})$.
Then $\{cc', b_1a_1\}$ is a $2$-matching and the removal of it
results in an odd component $G[V(O_{a_1})\backslash \{a_1, c'\}]$,
which contradicts Lemma \ref{next1}.

If $|A|=4$, then $|E(V(O_{a_1})\backslash \{a_1\}, S_{e_1})|=2$.
Let $A=\{b_1, b_2, b_3, b_4\}$. Without loss of generality, assume that
$b_1$ and $b_2$ have neighbors in $V(O_{a_1})\backslash \{a_1\}$.
Let $b_{1}'\in N(b_1)\cap (V(O_{a_1})\backslash \{a_1\})$.
Then $\{b_1b_{1}', b_2a_1\}$ is
a $2$-matching and the removal of it results in
an odd component $G[V(O_{a_1})\backslash \{a_1, b_{1}'\}]$,
which contradicts Lemma \ref{next1}.

\vspace{0.3cm}
\noindent {\bf Subcase 2.2.2. $|S_{e_1}|=5$.}
\vspace{0.3cm}

By Subcase \ref{subcase1},
$G_{e_1}$ has no Configuration as shown in Fig. \ref{tu-sub1} (a).
Note that $N(u)\cap N(v)$ contains an isolate vertex $a_1$.
If $G_{e_1}$ has Configuration as shown in Fig. \ref{tu-sub1} (b),
then $N(a_1)\cap N(a_2)$ contains an isolate vertex.
Since $|N(a_1) \cap N(a_2)|\geq 3$ and $G$ is claw-free,
$G[N(a_1)\cap N(a_2)]$ is a complete graph.
That is, $N(a_1)\cap N(a_2)$ has no isolate vertex.
Then $G_{e_1}$ has no Configuration as shown in Fig. \ref{tu-sub1} (b).
It follows that $G[\{a_1\}]$ and $G[\{a_2\}]$ are two trivial odd components of $G_{e_1}$.
Thus $\{u, v\}\subseteq S_{e_1}$ and $\{u, v\}\cap N(a_2)=\emptyset$.
So $N(a_2)\subseteq (S_{e_1}\backslash \{u, v\})\cup \{a_1\}$,
that is, $d_G(a_2)\leq 4$,
contradicting that $\delta(G)\geq 6$.

\vspace{0.3cm}

By Subcases \ref{subcase1} and \ref{subcase2}, $|V(O_1)|=1$ and $|V(O_2)|=1$.
Since $\delta(G)\geq 6$, $u$ and $v$ are adjacent to every vertex in $S_e$.
So (\romannumeral2) holds.
\end{proof}

By Lemma \ref{lem1}, we know that if $S_e$ is a smallest set satisfying
the conclusions of Theorem \ref{AC1994} when Theorem \ref{AC1994} is
applied to the edge $e$, then $S_e$ satisfies the conclusions of Lemma \ref{lem1}.
In subsequent discussions, we keep in mind that $S_e$ is always a certain smallest set
when Theorem \ref{AC1994} is applied to the edge $e$, and $G_{e}=G-e-S_{e}$.
Let $O_1, O_2$ and $O_3$ be the odd components of $G_e$ such that
$u \in V(O_1)$ and $v \in V(O_2)$.

\begin{defi}\label{def1}
Let $G$ be a minimal $2$-extendable claw-free graph with
$\delta(G)\geq 6$.
For any edge $e=uv\in E(G)$,
let $S_e$ be a smallest set satisfying the conclusions of Theorem $\ref{AC1994}$
when Theorem $\ref{AC1994}$ is applied to the edge $e$. Let $G_{e}=G-e-S_{e}$.
We say $e$ is of type $1$ if $G_e$ satisfies Statement {\rm (\romannumeral1)}
of Lemma $\ref{lem1};$ otherwise, $e$ is of type $2$.
\end{defi}

\begin{figure}[h]
\centering
\includegraphics[height=5cm,width=9.2cm]{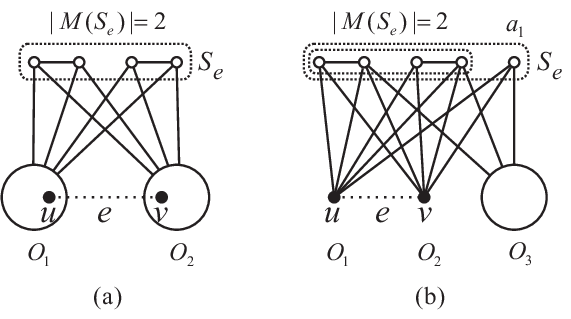}
\caption{\label{tu-def} (a) $e$ is of type 1 if $|S_e|=4$; (b) $e$ is of type 2 if $|S_e|=5$.}
\end{figure}

If $e$ is of type $1$, then $G$ has the configuration as shown in Fig. \ref{tu-def}\ (a);
if $e$ is of type $2$, then $G$ has the configuration as shown in Fig. \ref{tu-def}\ (b)
by Lemma \ref{lem1}. Thus we have the following proposition.

\begin{prop}\label{prop1}
Let $G$ be a minimal $2$-extendable claw-free graph with $\delta(G)\geq 6$.
For any edge $uv$ of $G$, $uv$ is either of type $1$ or of type $2$,
not both.
\end{prop}



Note that if $uv$ is of type 2, then
$N(u)=S_{uv} \cup \{v\}$ and $N(v)=S_{uv} \cup \{u\}$.
It follows that $N(u) \backslash \{v\}=N(v)\backslash\{u\}=S_{uv}$.
We state it as a proposition.

\begin{prop}\label{prop3}
Let $G$ be a minimal $2$-extendable claw-free graph with $\delta(G)\geq 6$.
If $uv\in E(G)$ and $uv$ is of type $2$,
then $N(u) \backslash \{v\}=N(v) \backslash\{u\}$.
\end{prop}

\begin{prop}\label{prop4}
Let $G$ be a minimal $2$-extendable claw-free graph with
$\delta(G)\geq 6$.
For any vertex $u$ of $G$, $u$ is incident with at most two edges of type $2$.
\end{prop}

\begin{proof}
Let $e=uv\in E(G)$ and $e$ be of type 2.
Let $S_e=\{x_1, y_1, x_2, y_2, a_1\}$ and $M(S_e)=\{x_1y_1, x_2y_2\}$.
Since $N(a_1)\cap V(O_3)\neq \emptyset$ and $N(u)\cap V(O_3)=\emptyset$,
$ua_1$ is not of type 2 by Proposition \ref{prop3}.
If $ux_1$ is of type 2,
then $N(u)\backslash\{x_{1}\}=N(x_{1})\backslash\{u\}$
$=(S_e\backslash \{x_1\})\cup \{v\}$.
So $a_1x_1\in E(G)$ and $N(x_{1}) \cap V(O_3)=\emptyset$.
Then $a_1y_1\notin E(G)$. Otherwise,
let $M'(S_e)=\{x_2y_2, a_1y_1\}$.
By Lemma \ref{next4}, $N(x_{1}) \cap V(O_3)\neq \emptyset$, a contradiction.
Thus $N(u)\backslash\{y_{1}\} \neq N(y_{1})\backslash\{u\}$.
By Proposition \ref{prop3}, $uy_1$ is not of type 2.

Suppose that $ux_2$ is of type 2.
By similar discussions above, $a_1x_2\in E(G)$, $N(x_{2}) \cap V(O_3)=\emptyset$
and $a_1y_2\notin E(G)$.
Since $G$ is claw-free, $y_1y_2\in E(G)$.
Then $\{y_1y_2, ua_1\}$ is a $2$-matching and
the removal of it results in an odd component $O_3$,
which contradicts Lemma \ref{next1}.
Moreover, $N(u)\backslash\{y_{2}\} \neq N(y_{2})\backslash\{u\}$.
By Proposition \ref{prop3}, $uy_2$ is not of type 2.
Thus, among all the edges associated with $u$,
only $ux_1$ may be of type 2 other than the edge $uv$.
The result follows.
\end{proof}

\begin{defi}\label{def2}
Let $G$ be a graph and $e=uv \in E(G)$.
We say a vertex cut $X_{e}^{u}$
$($resp. $X_{e}^{v}$$)$ of $G$ satisfies Property $P$ if

{\rm (\romannumeral1)} $|X_{e}^{u}|=5$ and $|M(X_{e}^{u})|=2$
$($resp. $|X_{e}^{v}|=5$ and $|M(X_{e}^{v})|=2$$);$

{\rm (\romannumeral2)} $v\in X_{e}^{u}$ $($resp. $u\in X_{e}^{v}$$);$

{\rm (\romannumeral3)} $u$ lies in an odd component, say $C_u$, of
$G-X_{e}^{u}$ $($resp. $v$ lies in an odd component, say $C_v$, of $G-X_{e}^{v}$$);$ and

{\rm (\romannumeral4)} $E(V(C_u), \{v\})=\{uv\}$
$($resp. $E(V(C_v), \{u\})=\{uv\}$$)$.
\end{defi}

Let $G$ be a minimal $2$-extendable claw-free graph with $\delta(G)\geq 6$.
Let $e=uv\in E(G)$. If $e$ is of type 1, then $S_e\cup \{v\}$ is a
$5$-vertex cut of $G$, $M(S_e\cup \{v\})=2$
and $u$ lies in the odd component $O_1$ of $G-(S_e\cup \{v\})$.
Moreover, $E(V(O_1), \{v\})=\{uv\}$.
These observations lead to the following lemma.

\begin{lem}\label{lem2}
Let $G$ be a minimal $2$-extendable claw-free graph with $\delta(G)\geq 6$.
If $uv\in E(G)$ and $uv$ is of type $1$,
then there exists a $5$-vertex cut $S_{uv}\cup \{v\}$
$($resp. $S_{uv}\cup \{u\}$$)$ of $G$ satisfying Property $P$.
\end{lem}

We are now ready to prove Theorem \ref{main}.

\bigskip
\noindent{\bf Proof of Theorem \ref{main}.}
By Theorem \ref{PL}, $\delta(G)\geq 4$.
Suppose to the contrary that $\delta(G)\geq 6$.
By Propositions \ref{prop1} and \ref{prop4},
every vertex of $G$ is incident with at least four edges of type 1.
Combining with Lemma \ref{lem2}, for every edge of type 1 of $G$,
there exists a $5$-vertex cut of $G$ satisfying Property $P$.

Among all $5$-vertex cuts of $G$ that satisfy Property $P$,
we choose one, denoted by $X_{e}^{u}$, such that
the odd component of $G-X_{e}^{u}$ containing $u$ is of the minimum order,
where $e=uv\in E(G)$ and $v\in X_{e}^{u}$.
Denote by $G_{1}$ the odd component of $G-X_{e}^{u}$ containing $u$.
Let $G_{2}=G-X_{e}^{u}-V(G_{1})$ (as shown in Fig. \ref{tu-main1} (a)).
Since both $|V(G_1)|$ and $|X_{e}^{u}|$ are odd, $|V(G_{2})|$ is even.

\begin{figure}[h]
\centering
\includegraphics[height=5cm,width=10cm]{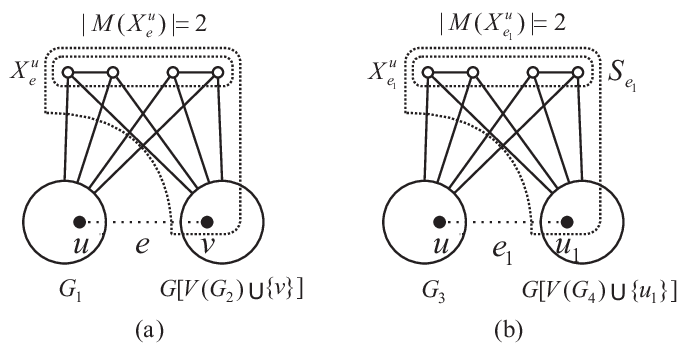}
\caption{\label{tu-main1} (a) $X_{e}^{u}$ satisfies Property $P$;
(b) $X_{e_1}^{u}$ satisfies Property $P$.}
\end{figure}

Since $|V(G_{1})|>1$, $N(u)\cap V(G_1)\neq \emptyset$.
Let $u_1\in N(u)\cap V(G_1)$ and $e_1=uu_1$.
Since $v\in N(u)$ and $v\notin N(u_1)$, $N(u) \backslash \{u_1\} \neq N(u_1)\backslash \{u\}$.
By Propositions \ref{prop1} and \ref{prop3}, $e_1$ is of type 1.
Then there exists $S_{e_1} \subseteq V(G)\backslash \{u, u_1\}$ with $|M(S_{e_1})|=2$
such that $G-e_1-S_{e_1}$ has exactly two odd components containing $u$ and $u_1$, respectively.
Let $X_{e_1}^{u}=S_{e_1} \cup \{u_1\}$.
By Lemma \ref{lem2}, $X_{e_1}^{u}$ satisfies Property $P$.
Let $G_{3}$ be the odd component of $G-X_{e_1}^{u}$ containing $u$ and
$G_{4}=G-X_{e_1}^{u}-V(G_{3})$ (as shown in Fig. \ref{tu-main1} (b)).
Since $G[V(G_4)\cup \{u_1\}]$ is an odd component of $G-e_1-S_{e_1}$,
$|V(G_{4})|$ is even.

Now we have two different partitions of $V(G)$ regarding $X_{e}^{u}$
and $X_{e_1}^{u}$, respectively. One is $\{V(G_1), X_{e}^{u}, V(G_2)\}$
and the other is $\{V(G_3), X_{e_1}^{u}, V(G_4)\}$.
We view the two partitions together as shown in Fig. \ref{tu-main2}.
For convenience, we introduce several notations.
Let $T=X_{e}^{u} \cap X_{e_1}^{u}$, $X_1=X_{e}^{u} \cap V(G_{3})$, $X_2=X_{e}^{u} \cap V(G_{4})$,
$Y_1=X_{e_1}^{u} \cap V(G_{1})$ and $Y_2=X_{e_1}^{u}\cap V(G_{2})$.
Then $X_{e}^{u}=X_1 \cup X_2 \cup T$ and $X_{e_1}^{u}=Y_1 \cup Y_2 \cup T$.
Let $H_{0}=V(G_{1})$ $\cap$ $V(G_{3})$, $H_{1}=V(G_{1}) \cap V(G_{4})$,
$H_{2}=V(G_{2}) \cap V(G_{4})$ and $H_{3}=V(G_{2}) \cap V(G_{3})$.
So $u \in V(G_{1}) \cap V(G_{3})=H_{0}$.
Since $u_1 \in V(G_{1})$ and $u_1\in X_{e_1}^{u}$, $u_1\in X_{e_1}^{u} \cap V(G_{1})=Y_1$.
Moreover, $v \in X_1 \cup T$ as $v\in X_{e}^{u}\cap N(u)$.

\begin{figure}[h]
\centering
\includegraphics[height=4cm,width=4.7cm]{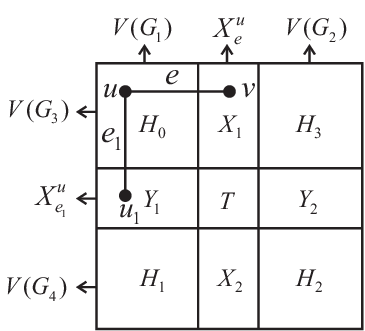}
\caption{\label{tu-main2} Illustration of the partition of $V(G)$
with respect to $X_{e}^{u}$ and $X_{e_1}^{u}$.}
\end{figure}

If $N(u) \cap H_{0}\neq \emptyset$, let $u_0 \in N(u) \cap H_{0}$.
Since $N(v) \cap V(G_{1})=\{u\}$ and $N(u_1)$ $\cap$ $V(G_{3})=\{u\}$,
$\{vu_0, vu_1, u_0u_1\} \cap E(G)=\emptyset$.
Then $G[\{u, v, u_0, u_1\}]$ is a claw, a contradiction. Thus
\begin{equation}\label{30}
N(u) \subseteq X_1 \cup Y_1 \cup T.
\end{equation}
By the hypothesis that $\delta(G) \geq 6$, we deduce that
\begin{equation}\label{31}
|X_1|+|Y_1|+|T|\geq d_{G}(u) \geq 6.
\end{equation}
Combining with $|X_{e}^{u}|+|X_{e_1}^{u}|=10$, we obtain that
\begin{equation}\label{32}
|X_2|+|Y_2|+|T| \leq 4.
\end{equation}
By (\ref{32}), $|Y_1|+|Y_2|+|T|=5>|X_2|+|Y_2|+|T|$.
It implies that
\begin{equation}\label{032}
|Y_1|>|X_2|.
\end{equation}
Furthermore, by (\ref{31}), $|X_1|+|Y_1|+|T|\geq 6>|Y_1|+|Y_2|+|T|$. Then
\begin{equation}\label{033}
|X_1|>|Y_2|.
\end{equation}

\begin{claim}\label{claim1}
$H_{2} \neq \emptyset$.
\end{claim}

\noindent Proof of Claim \ref{claim1}.
Suppose to the contrary that $H_{2}=\emptyset$. Then $V(G_{4})=H_{1} \cup X_2$.
So $G[V(G_{4})\cup \{u_1\}]=G[H_{1} \cup X_2 \cup \{u_1\}]$.
Since $G[V(G_{4})\cup \{u_1\}]$ is an odd component of $G_{e_1}$,
$G[H_{1} \cup X_2 \cup \{u_1\}]$ is connected. For the edge $uu_1$,
there is a $5$-vertex cut $(X_{e_1}^{u} \backslash \{u_1\}) \cup \{u\}$
such that $G[H_{1} \cup X_2 \cup \{u_1\}]$ is an odd component
of $G-((X_{e_1}^{u} \backslash \{u_1\}) \cup \{u\})$ containing $u_1$.
Moreover, $u$ has the unique neighbor $u_1$ in $G[H_{1} \cup X_2 \cup \{u_1\}]$
and $|M((X_{e_1}^{u} \backslash \{u_1\}) \cup \{u\})|=2$.
So $(X_{e_1}^{u} \backslash \{u_1\}) \cup \{u\}$ satisfies Property $P$.
By (\ref{032}),
$|V(G[H_{1} \cup X_2 \cup \{u_1\}])| \leq |V(G[H_{1} \cup Y_1])|<|V(G_{1})|$,
a contradiction to the minimality of $V(G_{1})$.
\hfill $\square$

\begin{claim}\label{claim2}
$X_2 \neq \emptyset$.
\end{claim}

\noindent Proof of Claim \ref{claim2}. Suppose to the contrary that $X_2=\emptyset$.
Then $G[V(G_{4}) \cup \{u_1\}]=G[H_{1} \cup H_{2} \cup \{u_1\}]$.
Since $G[V(G_{4}) \cup \{u_1\}]$ is an odd
component of $G_{e_1}$, $G[V(G_{4}) \cup \{u_1\}]$ is connected.
But $G[H_{1} \cup H_{2} \cup \{u_1\}]$ is disconnected as there is no path connecting
$H_{1} \cup \{u_1\}$ and $H_{2}$ through a vertex in $X_2$, a contradiction.
\hfill $\square$

\begin{claim}\label{claim401}
If $N(u_1)\cap H_1=\emptyset$ $($resp. $N(v)\cap H_3=\emptyset$$)$,
then $|X_2|+|Y_1|+|T|\geq 6$ $($resp. $|X_1|+|Y_2|+|T|\geq 6$$)$.
\end{claim}

\noindent Proof of Claim \ref{claim401}.
Suppose to the contrary that $|X_2|+|Y_1|+|T|<6$.
By hypothesis,
$N(u_1)\subseteq \{u\}\cup X_2\cup (Y_1\backslash \{u_1\})\cup T$.
That is, $d_G(u_1)<6$, contradicting the assumption that $\delta(G)\geq 6$.
Similarly, $|X_1|+|Y_2|+|T|\geq 6$.
\hfill $\square$

\begin{claim}\label{claim003}
For every edge in $M(X_{e}^{u})$, at least one end-vertex of it has a neighbor in $V(G_1)$.
\end{claim}

\noindent Proof of Claim \ref{claim003}.
Let $M(X_{e}^{u})=\{x_1y_1, x_2y_2\}$.
Suppose that $N(x_1)\cap V(G_1)=\emptyset$ and $N(y_1)\cap V(G_1)=\emptyset$.
Let $v'\in N(v)\cap V(G_2)$. Then $\{x_2y_2, vv'\}$ is
a $2$-matching and the removal of it results in an odd component $G_1$.
By Lemma \ref{next1}, $G$ is not $2$-connected, a contradiction.
\hfill $\square$

\vspace{0.2cm}

If there is a vertex $x$ in $T$ such that at least
three of $H_0$, $H_1$, $H_2$ and $H_3$ contain vertices in $N(x)$,
then $G$ has a claw, a contradiction. This fact yields the following claim.

\begin{claim}\label{claim500}
Every vertex in $T$ has neighbors in at most two of $H_0$, $H_1$, $H_2$
and $H_3$.
\end{claim}

\begin{claim}\label{claim3}
$Y_2\neq \emptyset$.
\end{claim}

\noindent Proof of Claim \ref{claim3}.
Suppose to the contrary that $Y_2=\emptyset$.
Since $N(v)\cap H_2=\emptyset$ and $H_2\neq \emptyset$,
$X_{e}^{u}\backslash \{v\}$ is a vertex cut of $G$.
Since $G[X_{e}^{u}\backslash \{v\}]$
has a perfect matching with size $2$, by Lemma \ref{next3},
$|H_2|$ is even and every vertex in $X_{e}^{u}\backslash \{v\}$
has a neighbor in $H_2$. So $X_1=\{v\}$.
Note that $|X_1|+|Y_2|+|T|=|X_1|+|T|<6$. By Claim \ref{claim401},
$N(v)\cap H_3\neq \emptyset$. So $H_3\neq \emptyset$.
Since $G$ is $3$-connected, $T$ contains
at least two vertices which have neighbors in $H_3$.
By Claim \ref{claim500},
$T$ contains at least two vertices which are not adjacent to $u$.
Since $N(u)\subseteq Y_1\cup T\cup X_1$ and
$|Y_1\cup T\cup X_1|=|Y_1\cup T\cup \{v\}|=6$, $d_G(u)\leq 4$, a contradiction.
\hfill $\square$

\begin{claim}\label{claim4}
$T\neq \emptyset$.
\end{claim}

\noindent Proof of Claim \ref{claim4}.
Suppose to the contrary that $T=\emptyset$.
By (\ref{31}), $|X_1\backslash \{v\}|+|Y_1\backslash \{u_1\}|\geq 4$.
Since $|M(X_{e}^{u}\backslash \{v\})|=2$ and $|M(X_{e_1}^{u}\backslash \{u_1\})|=2$,
$G[(X_1\backslash \{v\})\cup (Y_1\backslash \{u_1\})]$ has a perfect matching $M_1$
and $|M_1|\geq 2$. By (\ref{32}), $|X_2|+|Y_2|\leq 4$.
Then $G[X_2\cup Y_2]$ has a perfect matching $M_2$ and $|M_2|\leq 2$.
By Claim \ref{claim1}, $X_2\cup Y_2$ is a vertex cut of $G$.
Combining with Lemma \ref{next3},
$|M_2|\geq 2$. Then $|M_2|=2$ and $|M_1|=2$. So $|H_2|$ is even.
Thus $|X_2|+|Y_2|=4$ and $|X_1|+|Y_1|=6$.

Since $|V(G_4)|$, $|H_2|$ and $|X_2|$ are even, $|H_1|$ is even.
Because $V(G_1)=H_0\cup H_1\cup Y_1$ and $|V(G_1)|$, $|Y_1|$ are odd,  $|H_0\backslash \{u\}|$ is odd.
Noticing that $u, v$ and $u_1$ have no neighbors in $H_0\backslash \{u\}$,
$G-V(M_1)$ has an odd component $G[H_0\backslash \{u\}]$.
By Lemma \ref{next1}, $G$ is not $2$-extendable, a contradiction.
\hfill $\square$

\begin{claim}\label{claim5}
If $G[X_2\cup Y_2\cup T]$ has a perfect matching $M_2$,
then $M_2\nsubseteq M(X_{e}^{u})\cup M(X_{e_1}^{u})$.
\end{claim}

\noindent Proof of Claim \ref{claim5}.
Suppose to the contrary that $M_2\subseteq M(X_{e}^{u})\cup M(X_{e_1}^{u})$.
Take any vertex $x$ in $T$. We know that $x$ is saturated by $M_2$.
If $x$ is matched to a vertex in $Y_2$ by $M(X_{e_1}^{u})$,
then $x$ is matched to a vertex in $X_1$ by $M(X_{e}^{u})$.
If $x$ is matched to a vertex in $X_2$ by $M(X_{e}^{u})$,
then $x$ is matched to a vertex in $Y_1$ by $M(X_{e_1}^{u})$.
Thus $x$ can not be matched to a vertex in $X_1$ (resp. $X_2$)
by $M(X_{e}^{u})$ and a vertex in $Y_1$ (resp. $Y_2$) by $M(X_{e_1}^{u})$.
So $G[((X_1\cup T)\backslash \{v\})\cup (Y_1\backslash \{u_1\})]$
has a perfect matching $M_1$
and $M_1\subseteq M(X_{e}^{u})\cup M(X_{e_1}^{u})$.
Since $|X_2\cup Y_2 \cup T|$ and
$|(Y_1\backslash \{u_1\}) \cup Y_2 \cup T|$ are even,
$|X_2|$ and $|Y_1|$ have different parity.

By Claim \ref{claim1}, $X_2\cup Y_2\cup T$ is a vertex cut of $G$.
By Lemma \ref{next3}, $|M_2|\geq 2$.
Since $|M(X_{e}^{u})|+|M(X_{e_1}^{u})|=|M_1|+|M_2|=4$, $|M_1|\leq 2$.
By (\ref{31}),
$|((X_1\cup T)\backslash \{v\})\cup (Y_1\backslash \{u_1\})|\geq 4$.
Then $|M_1|\geq 2$. So $|M_1|=2$ and $|M_2|=2$.

By Lemma \ref{next3}, $|H_2|$ is even. Note that $V(G_4)=H_1\cup H_2\cup X_2$ and $|V(G_4)|$ is even.
If $|X_2|$ is even, then $|Y_1|$ is odd and $|H_1|$ is even.
Since $V(G_1)=H_0\cup H_1\cup Y_1$ and $|V(G_1)|$ is odd, $|H_0|$ is even.
Similarly, if $|X_2|$ is odd, then $|H_0|$ is even.
In either case, $|H_0\backslash \{u\}|$ is odd.
Since $N(u_1)\cap (H_0\cup X_1)=\{u\}$ and $N(v)\cap (H_0\cup Y_1)=\{u\}$,
$N(u_1)\cap (H_0\backslash \{u\})=\emptyset$ and $N(v)\cap (H_0\backslash \{u\})=\emptyset$.
Thus $G-V(M_1)$ has an odd component $G[H_0\backslash \{u\}]$.
By Lemma \ref{next1}, $G$ is not $2$-extendable, a contradiction.
\hfill $\square$

\bigskip
By Claim \ref{claim5}, there exists a vertex $x$ in $T$ such that
$x$ satisfies at least one of the following two conditions:

(\romannumeral1) $x$ is matched to a vertex in $X_1$ by $M(X_{e}^{u})$
and a vertex in $Y_1$ by $M(X_{e_1}^{u})$;
\hfill $(\dag)$

(\romannumeral2) $x$ is matched to a vertex in $X_2$ by $M(X_{e}^{u})$
and a vertex in $Y_2$ by $M(X_{e_1}^{u})$.
\hfill $(\ddag)$




\smallskip
In the sequel, the red (resp. blue) edges represent the edges
in $M(X_{e}^{u})$ (resp. $M(X_{e_1}^{u})$).
By (\ref{032}), (\ref{033}), Claims \ref{claim2} and \ref{claim3},
we obtain that $|Y_1|>|X_2|$, $|X_1|>|Y_2|$,
$X_2 \neq \emptyset$ and $Y_2 \neq \emptyset$.
Since $v\in X_1\cup T$, we divide the proof into the following two cases.

\begin{case}\label{case011}
$v\in X_1$.
\end{case}

If the vertices in $T$ only satisfy the condition $(\dag)$,
then $|X_2|=2$ and $|Y_2|=2$.
But $|Y_1|=2$, a contradiction to (\ref{032}).
Thus the vertices in $T$ only satisfy the condition $(\ddag)$
(as shown in Fig. \ref{tu-main3} (a))
or satisfy conditions $(\dag)$ and $(\ddag)$
(as shown in Fig. \ref{tu-main3} (b)).

\begin{figure}[h]
\centering
\includegraphics[height=4.5cm,width=9cm]{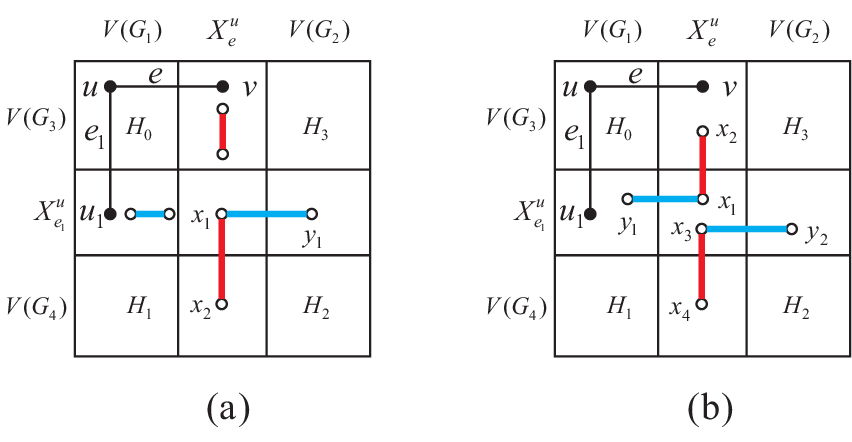}
\caption{\label{tu-main3} (a) The vertices in $T$ only satisfy
the condition $(\ddag)$;}
{(b) The vertices in $T$ satisfy conditions $(\dag)$
and $(\ddag)$.}
\end{figure}

(a) Since $G[H_1\cup H_2\cup \{x_2, u_1\}]=G[V(G_4)\cup \{u_1\}]$ is connected,
$N(x_2)\cap H_1 \neq \emptyset$ and $N(x_2)\cap H_2\neq \emptyset$.
Note that $\{x_1, x_2, y_1\}$ is a 3-vertex cut of $G$ and
$\{x_1y_1\}$ is a 1-matching in $G[\{x_1, x_2, y_1\}]$.
By Lemma \ref{next2}, $G$ is not $2$-extendable, a contradiction.

(b) In this case, $|X_1|+|Y_1|+|T|=6$. Combining (\ref{30})
and $\delta(G)\geq 6$, $N(u)=\{u_1, v, x_1, x_2, x_3, y_1\}$.
If $N(x_1)\cap H_2=\emptyset$, then $\{x_3, x_4, y_2\}$ is a 3-vertex cut of $G$.
The proof is similar to the case (a).
So $N(x_1)\cap H_2\neq \emptyset$.
By Claim \ref{claim500}, $N(x_1)\cap H_1=\emptyset$.
Moreover, $N(x_3)\cap H_1=\emptyset$.
Otherwise, let $x_{3}'\in N(x_3)\cap H_1$.
Then $G[\{x_3, u, y_2, x_{3}'\}]$ is a claw, a contradiction.
Since $|X_2|+|Y_1|+|T|=5$,
$N(u_1)\cap H_1\neq \emptyset$ by Claim \ref{claim401}.
So $H_1\neq \emptyset$. Then $\{u_1, y_1, x_4\}$ is a 3-vertex cut of $G$.
Because $G[V(G_4)\cup \{u_1\}]=G[H_1\cup H_2\cup \{x_4, u_1\}]$ is
connected, $N(x_4)\cap H_1\neq \emptyset$ and $N(x_4)\cap H_2\neq \emptyset$.
Note that $\{u_1y_1\}$ is a 1-matching in $G[\{u_1, y_1, x_4\}]$.
By Lemma \ref{next2}, $G$ is not $2$-extendable, a contradiction.

\begin{case}\label{case012}
$v\in T$.
\end{case}

Since $N(v)\cap V(G_1)=\{u\}$, $v$ is matched to a vertex in $T$ or
$Y_2$ by $M(X_{e_1}^{u})$. Let $vv_1\in M(X_{e_1}^{u})$.
Suppose that $v_1\in T$.
If $v_1$ is matched to a vertex in $T$ by $M(X_{e}^{u})$,
then $X_1=\emptyset$ as $X_2\neq \emptyset$. By Claim \ref{claim3},
$Y_2\neq \emptyset$, a contradiction to (\ref{033}).
If $v_1$ is matched to a vertex in $X_1$ by $M(X_{e}^{u})$,
then $Y_1=\{u_1\}$ as $Y_2\neq \emptyset$. By Claim \ref{claim2},
$X_2\neq \emptyset$, a contradiction to (\ref{032}).
If $v_1$ is matched to a vertex in $X_2$ by $M(X_{e}^{u})$,
then $|X_2|\geq 1$. By Claim \ref{claim3}, $Y_2\neq \emptyset$.
So $Y_1=\{u_1\}$, a contradiction to (\ref{032}).

Suppose that $v_1\in Y_2$. Then $|X_1|\geq 2$
by (\ref{033}). Since $X_2\neq \emptyset$, $|X_2|=1$ or $|X_2|=2$.
By (\ref{032}), $|X_2|<|Y_1|$.
It follows that $|X_2|=2$ and $|Y_1|=3$
(as shown in Fig. \ref{tu-main4} (c))
or $|X_2|=1$ and $|Y_1|=2$ (as shown in Fig. \ref{tu-main4} (d)).
Both of the two cases contradict Claim \ref{claim5}.

\begin{figure}[h]
\centering
\includegraphics[height=4.5cm,width=9.5cm]{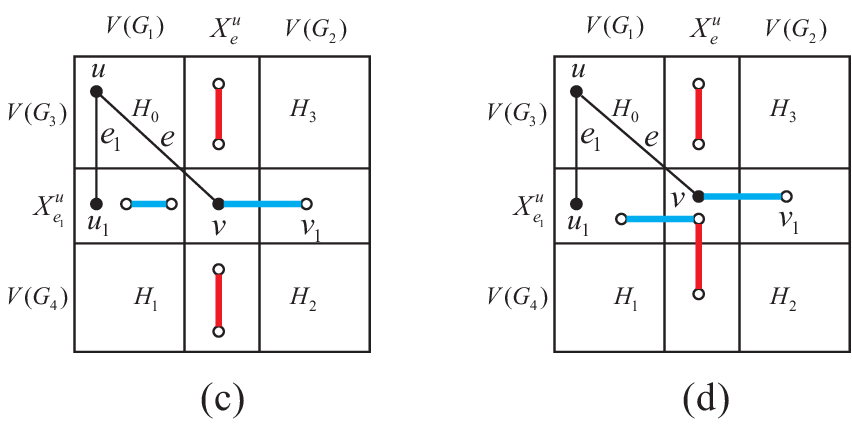}
\caption{\label{tu-main4} (c) $|X_2|=2$ and $|Y_1|=3$;
(d) $|X_2|=1$ and $|Y_1|=2$.}
\end{figure}

From above discussions, it is impossible that $\delta(G) \geq 6$.
Therefore, $\delta(G)=4$ or $5$.
\hfill $\square$

\bigskip

Finally, we have shown that the minimum degree of a minimal
3-extendable claw-free graph is either 6 or 7.
Thus we make the following conjecture for further investigation.

\begin{conj}\label{conj}
The minimum degree of a minimal $k$-extendable claw-free graph is
either $2k$ or $2k+1$.
\end{conj}

\end{document}